\documentclass[11pt, oneside, english]{article}

\usepackage[english]{babel}
\usepackage[utf8x]{inputenc}
\usepackage[T1]{fontenc}
\usepackage[letterpaper,top=1in,bottom=1in,left=1in,right=1in,marginparwidth=1.75cm]{geometry}
\usepackage{dsfont}
\usepackage{amsmath, amssymb, amsthm, verbatim,enumerate,bbm}
\usepackage{indentfirst, bbm}
\usepackage{appendix}
\usepackage{enumerate}
\usepackage{verbatim}
\usepackage[colorinlistoftodos]{todonotes}
\usepackage[colorlinks=true, allcolors=blue]{hyperref}
\usepackage[nameinlink,capitalise,noabbrev]{cleveref}
\crefname{appsec}{Appendix}{Appendices}
\creflabelformat{enumi}{#2textup{#1}#3}

\allowdisplaybreaks
\makeatletter
\newtheorem{theorem}{Theorem}[section]

\newtheorem*{namedtheorem}{\theoremname}
\newcommand{\theoremname}{testing}

\newtheorem{lemma}[theorem]{Lemma}

\newtheorem{claim}[theorem]{Claim}

\newtheorem{proposition}[theorem]{Proposition}
\newtheorem{observation}[theorem]{Observation}

\newtheorem*{question*}{Question}

\theoremstyle{definition}
\newtheorem{definition}[theorem]{Definition}

\newtheorem{remark}[theorem]{Remark}

\theoremstyle{plain}

\makeatother

\title{The strong circular law: a combinatorial view}
\author{
Vishesh Jain\thanks{Massachusetts Institute of Technology. Department of Mathematics. Email: {\tt visheshj@mit.edu}. 
}}

\DeclareMathOperator{\LCD}{LCD}

\date{}
\begin{document}
\maketitle

\global\long\def\R{\mathbb{R}}

\global\long\def\S{\mathbb{S}}

\global\long\def\Z{\mathbb{Z}}

\global\long\def\C{\mathbb{C}}

\global\long\def\Q{\mathbb{Q}}

\global\long\def\N{\mathbb{N}}

\global\long\def\P{\mathbb{P}}

\global\long\def\F{\mathbb{F}}

\global\long\def\U{\mathcal{U}}

\global\long\def\V{\mathcal{V}}

\global\long\def\E{\mathbb{E}}

\global\long\def\Ev{\mathscr{Rk}}

\global\long\def\Dg{\mathscr{D}}

\global\long\def\Ndg{\mathscr{ND}}

\global\long\def\Rv{\mathcal{R}}

\global\long\def\Gv{\mathscr{Null}}

\global\long\def\Hv{\mathscr{Orth}}

\global\long\def\Supp{{\bf Supp}}

\global\long\def\Sv{\mathscr{Spt}}

\global\long\def\ring{\mathfrak{R}}

\global\long\def\1{\mathbbm{1}}

\global\long\def\Bad{{\boldsymbol{B}}}

\global\long\def\supp{{\bf supp}}

\global\long\def\A{\mathcal{A}}

\global\long\def\L{\mathcal{L}}

\global\long\def\dist{\text{dist}}

\begin{abstract}
Let $N_n$ be an $n\times n$ complex random matrix, each of whose entries is an independent copy of a centered complex random variable $z$ with finite non-zero variance $\sigma^{2}$. The strong circular law, proved by Tao and Vu, states that almost surely, as $n\to \infty$, the empirical spectral distribution of $N_n/(\sigma\sqrt{n})$ converges to the uniform distribution on the unit disc in $\C$. 

A crucial ingredient in the proof of Tao and Vu, which uses deep ideas from additive combinatorics, is controlling the lower tail of the least singular value of the random matrix $xI - N_{n}/(\sigma\sqrt{n})$ (where $x\in \C$ is fixed) with failure probability that is inverse polynomial. In this paper, using a simple and novel approach (in particular, not using tools from additive combinatorics or any net arguments), we show that for any fixed matrix $M$ with operator norm at most $n^{3/4 - \epsilon}$ and for all $\eta \geq 0$,
$$\Pr\left(s_n(M+N_n) \leq \eta \right) \lesssim n^{C}\eta + \exp(-n^{c}),$$
where $s_n(M+N_n)$ is the least singular value of $M+N_n$ and $C,c$ are absolute constants. Our result is optimal up to the constants $C,c$ and the inverse exponential-type error rate improves upon the inverse polynomial error rate due to Tao and Vu. 

Our proof relies on an extension of the solution to the so-called counting problem in inverse Littlewood--Offord theory, developed by Ferber, Luh, Samotij, and the author, along with a novel `rounding trick' based on controlling the $\infty \to 2$ operator norm of heavy-tailed random matrices.

\end{abstract}

\section{Introduction}
Let $N_n$ be an $n\times n$ complex random matrix, each of whose entries is an independent copy of a complex random variable $z$ with mean $0$ and finite non-zero variance $\sigma^{2}$. The \emph{empirical spectral distribution} (ESD) $\mu_{n}$ of $N_n$ is defined on $\R^{2}$ by the expression
$$\mu_{n}(s,t) := \frac{1}{n}\cdot \left|\{k \in [n] \mid \Re(\lambda_k) \leq s; \Im(\lambda_k) \leq t\}\right|,$$
where $\lambda_1,\dots,\lambda_n$ denote the eigenvalues of $N_n/\sigma\sqrt{n}$. The celebrated \emph{strong circular law} of Tao and Vu \cite{tao2010random} asserts that almost surely, as $n$ tends to infinity, $\mu_{n}$ converges uniformly to 
$$\mu_{\infty}(s,t):= \frac{1}{\pi}\text{area}\{x\in \C \mid |x| \leq 1, \Re(x) \leq s, \Im(x) \leq t\}.$$ 
The circular law has a long history dating back to the 1950s when it was conjectured as a natural non-Hermitian counterpart to Wigner's famous semi-circle law, and prior to Tao and Vu's definitive work, many researchers obtained partial results requiring extra distributional assumptions on the random variable $z$, and very often weakening the notion of convergence from almost sure convergence to convergence in probability (this is not just a technical point, and genuinely new ideas are required to obtain almost sure convergence; see the discussion in Section 2 of \cite{tao2008random}). We refer the reader to the survey \cite{bordenave2012around} and the references therein for a much more detailed discussion of the history of this problem.

In the case when we further assume that $z$ has $2+\eta$ moments for some $\eta > 0$, the approach of Bai \cite{bai1997circular}, Bai and Silverstein \cite{bai2010spectral}, and Girko \cite{girko1985circular} reduces the problem to controlling the lower tail of the least singular value of the random matrix $xI - N_n/(\sigma\sqrt{n})$, where $x\in \C$ is fixed; even when we assume that $z$ has only non-zero finite variance, controlling the lower tail of the least singular value of this random matrix is a fundamental step in Tao and Vu's proof (recall that the least singular value of a complex matrix $M_n$, denoted by $s_n(M_n)$, is the smallest eigenvalue of the positive semidefinite matrix $\sqrt{M_{n}^{\dagger}M_{n}}$). To this end, Tao and Vu \cite{tao2008random} showed using sophisticated techniques from additive combinatorics that for any constants $A,C > 0$, there exists a constant $B > 0$ such that for any $n\times n$ fixed (complex) matrix $M$ of operator norm at most $n^{C}$,
\begin{equation}
\label{eqn:TV}
    \Pr\left(s_n(M + N_n) \leq n^{-B}\right) \lesssim n^{-A}.
\end{equation}
The dependence of $B$ on $A$ and $C$ can be made explicit and was subsequently sharpened in \cite{tao2010smooth}. Note that for the proof of the circular law, fixing $C = O(\sqrt{n})$ is sufficient.\\

Our goal in the present work is to provide a simple and elementary proof of a quantitative strengthening of \cref{eqn:TV} in the setting of the strong circular law. More precisely, we show:
\begin{theorem}
\label{thm:main}
Let $z$ be a complex random variable with mean $0$ and variance $1$ and let $N_n$ be an $n\times n$ random matrix, each of whose entries is an independent copy of $z$. Let $M$ be a fixed complex matrix with operator norm at most $n^{0.51}$. Then, for all $\eta \geq 0$,
$$\Pr\left(s_n(M+N_n) \leq \eta \right) \leq C\left(n^{5/2}\eta + \exp(-cn^{1/50})\right),$$
where $C,c $ are constants depending only on $z$.
\end{theorem}

\begin{remark}
(1) In the above theorem, the choice of the power $n^{0.51}$ is arbitrarily made for convenience and could be replaced by $n^{0.75-\epsilon}$ for any $\epsilon > 0$; in follow-up work of the author \cite{jain2019c} which builds on some of the ideas in this paper, we will show (using a more complicated proof) how to obtain a bound on the lower tail of $M+N_n$ even if $\|M\| = O(\exp(n^{c}))$.  

(2) We have not tried to optimize any of the constants $C,c, 5/2, 1/50$, but note here that with additional work, the constant $5/2$ can be replaced by the nearly optimal value $1/2 + \epsilon$ for any $\epsilon > 0$ in the case when $\|M\| = O(\sqrt{n})$. Compared to \cref{eqn:TV}, our bound is closer to (optimal) bounds of the form
$$C\sqrt{n}\eta + C\exp(-cn),$$
which have been obtained under stronger assumptions: for the case when $z$ is a real subgaussian random variable and $\|M\| = O(\sqrt{n})$ in the landmark work of Rudelson and Vershynin \cite{rudelson2008littlewood}, and for the case when $z$ is a real random variable and (much more restrictively) $M = 0$ by Rebrova and Tikhomirov \cite{rebrova2018coverings}.

(3) We impose no constraints on the relationship between the real and imaginary parts of $z$ (for instance, in \cite{luh2018complex}, the real and imaginary parts are required to be i.i.d. subgaussian). In this generality, our result (along with a follow-up result of the author \cite{jain2019c}) is the only one showing that a random matrix, each of whose entries is an independent copy of a complex random variable of mean $0$ and variance $1$, is singular with probability at most $\exp(-n^{c})$ for some $c > 0$. Previously, this was not known, even if we further assume that the complex random variables are subgaussian -- the geometric machinery, pioneered by Rudelson and Vershynin \cite{rudelson2008littlewood}, runs into the obstacle that the real dimension of the complex unit sphere is $2n-1$ (see the discussion in Section 10 of \cite{rudelson2016no}).  
\end{remark}
Apart from the quantitative strengthening of \cref{eqn:TV}, we believe that our result is also interesting for the simplicity of the proof techniques, making use only of some standard Fourier analytic techniques along with elementary combinatorial ideas. In particular, in contrast to previous works in this area, we make no use of tools from additive combinatorics or net arguments. Parts of our proof which we believe may be of independent interest are the complex anti-concentration inequality \cref{thm:LCD-controls-sbp} and \cref{prop:reduction-to-integer-heavy-tailed}, which shows how to bypass the lack of control over the $2 \to 2$ operator norm in our setting by means of the $\infty\to 2$-norm. We hope that some of the ideas introduced in this work can aid in proving strong circular laws in other contexts such as \cite{basak2018circular, cook2017circular} where only weak circular laws are known so far.\\

\noindent {\bf Organization: }The rest of this paper is organized as follows. In \cref{sec:tools}, we collect some auxiliary results needed for the proof of our main theorem -- the key results here are \cref{thm:LCD-controls-sbp}, \cref{thm:counting-continuous}, and \cref{prop:operator-norm-control} (proved in \cref{sec:proof-operator-norm}). In \cref{sec:proof-main}, we prove \cref{thm:main} by combining these results. The key ingredient there is \cref{prop:reduction-to-integer-heavy-tailed}.\\

\noindent {\bf Notation: } Throughout the paper, we will omit floors and ceilings when they make no essential difference. For convenience, we will also say `let $p = x$ be a prime', to mean that $p$ is a prime between $x$ and $2x$; again, this makes no difference to our arguments. We will use $\S^{2n-1}$ to denote the set of unit vectors in $\C^{n}$, $B(x,r)$ to denote the ball of radius $r$ centered at $x$, and $\Re(\boldsymbol{v}), \Im(\boldsymbol{v})$ to denote the real and imaginary parts of a complex vector $\boldsymbol{v}\in \C^{n}$. As is standard, we will use $[n]$ to denote the discrete interval $\{1,\dots,n\}$. We will also use the asymptotic notation $\lesssim, \gtrsim, \ll, \gg$ to denote $O(\cdot), \Omega(\cdot), o(\cdot), \omega(\cdot)$ respectively. For a matrix $M$, we will use $\|M\|$ to denote its standard $\ell^{2}\to \ell^{2}$ operator norm. All logarithms are natural unless noted otherwise.  
\\

\noindent {\bf Acknowledgements: }I am indebted to Nick Cook for the suggestion to consider the complex setting, as well as very helpful discussions around the circular law. I would also like to thank Kyle Luh for pointing out that a statement similar to \cref{prop:operator-norm-control} also appears in  \cite{rebrova2018norms}, Hoi Nguyen for discussing his work \cite{nguyen2011optimal}, and Elizaveta Rebrova for discussing her work \cite{rebrova2018coverings}.

\section{Tools and auxiliary results}
\label{sec:tools}
In this section, we collect some preliminary results which will be used in the proof of \cref{thm:main}.

\subsection{Anti-concentration}
The goal of the theory of anti-concentration is to obtain upper bounds on the L\'evy concentration function, defined as follows.  

\begin{definition}[L\'evy concentration function]Let $z$ be an arbitrary complex random variable and let $\boldsymbol{v}:=(v_{1},\dots,v_{n})\in\C^{n}$. We define the \emph{L\'evy concentration function of $\boldsymbol{v}$ at radius $r$ with respect to $z$} by $$\rho_{r,z}(\boldsymbol{v}):=\sup_{x\in\C}\Pr\left(v_{1}z_{1}+\dots+v_{n}z_{n}\in B(x,r)\right),$$
where $z_1,\dots,z_n$ are independent copies of $z$.
\end{definition}
\begin{remark}
In particular, note that $\rho_{r,z}(1) = \sup_{x\in \C}\Pr(z \in B(x,r))$. We will use this notation repeatedly. 
\end{remark}

The next lemma shows that weighted sums of random variables with finite non-zero variance are not too close to being a constant. 
\begin{lemma}(see, e.g., Lemma 6.3 in \cite{tao2010smooth}) 
\label{lemma:anticoncentration}
Let $z$ be a  complex random variable with finite non-zero variance. Then, there exists a constant $c_{\ref{lemma:anticoncentration}} \in (0,1)$ depending only on $z$ such that for any $\boldsymbol{v}\in \S^{2n-1}$,
$$\sup_{\boldsymbol{v}\in \S^{n-1}}\rho_{c_{\ref{lemma:anticoncentration}},z}(\boldsymbol{v}) \leq 1-c_{\ref{lemma:anticoncentration}}.$$
\end{lemma}

Combining this with the so-called tensorization lemma (see Lemma 2.2 in \cite{rudelson2008littlewood}), we get the following standard estimate for `invertibility with respect to a single vector'.

\begin{lemma}
\label{lemma:invertibility-single-vector}
Let $z$ be a complex random variable with finite non-zero variance. 
Let $M$ be an arbitrary $n\times n$ complex matrix and let $N_n$ be an $n\times n$ complex random matrix each of whose entries is an independent copy of $z$. Then, for any fixed $\boldsymbol{v}\in \S^{2n-1}$,
$$\Pr\left(\|(M+N_{n})\boldsymbol{v}\|_2\leq c_{\ref{lemma:invertibility-single-vector}}\sqrt{n}\right) \leq (1- c_{\ref{lemma:invertibility-single-vector}})^{n},$$
where $c_{\ref{lemma:invertibility-single-vector}}\in (0,1)$ is a constant depending only on $z$. 
\end{lemma}

The next classical lemma, due to Esseen, is a generalization (up to constants) of the Erd\H{o}s-Littlewood-Offord anti-concentration inequality.
\begin{lemma}[Theorem 2 in \cite{esseen1966kolmogorov}] 
\label{thm:rogozin}
Let $z_1,\dots,z_n$ be jointly independent complex random variables and let $t_1,\dots,t_n$ be some positive real numbers. Then, for any $t \geq \max_j t_j$, we have
$$\rho_{t,\sum_{j=1}^{n}z_j}(1) \leq C_{\ref{thm:rogozin}}t^{2}\left(\sum_{j=1}^{n}t_j^{4}(1-\rho_{t_j, z_j}(1))\right)^{-1/2},$$
where $C_{\ref{thm:rogozin}}\geq 1$ is an absolute constant. 
\end{lemma}

The next definition isolates a convenient property of the random variables we consider in this paper. 
\begin{definition}
We say that a complex random variable $z$ is \emph{$C$-good} if 
\begin{equation}
\label{eqn:assumption-on-z}
\Pr(C^{-1}\leq|z_{1}-z_{2}|\leq C)\geq C^{-1},
\end{equation}
where $z_1$ and $z_2$ denote independent copies of $z$. 
The smallest $C\geq 1$ with respect to which $z$ is $C$-good will be denoted by $C_z$. 
\end{definition}

Indeed, it is straightforward to see that complex random variables with finite non-zero variance are $C$-good for some finite $C$, so that there is no loss of generality for us in imposing this additional restriction. 

\begin{observation}
\label{lemma:non-trivial-implies-good}
Let $z$ be a complex random variable with variance $1$. Then, $z$ is $C_{z}$-good for some $C_{z} \geq 1$. 
\end{observation}

We conclude this subsection with the following consequence of \cref{thm:rogozin}.

\begin{lemma}
\label{claim:anticonc-W}
Let $z$ be a complex random variable with variance $1$. There exists a constant $C_{\ref{claim:anticonc-W}}\geq 1$ depending only on $z$ such that for all $\boldsymbol{w}:=(w_1,\dots,w_n)\in (\Z+i\Z)^{n}$ with support of size at least $n^{0.99}$, $$\rho_{1,z}(\boldsymbol{w}) \leq C_{\ref{claim:anticonc-W}}n^{-0.495}.$$
\end{lemma}
\begin{proof}
As above, we know that $\rho_{v_z,z}(1)\leq u_z$ for some $u_z,v_z\in (0,1)$. Therefore, for all $j\in \supp(\boldsymbol{w})$,
$$\rho_{v_{z},w_{j}z_{j}}(1) \leq \rho_{|w_j|v_{z},w_{j}z_{j}}(1) \leq \rho_{v_{z},z_{j}}(1) \leq u_{z}.$$
Hence, by \cref{thm:rogozin}, 
$$\rho_{v_{z},\sum_{j=1}^{n}w_j z_{j}}(1) \leq \frac{C_{\ref{thm:rogozin}}}{\sqrt{|\supp(\boldsymbol{w})|(1-u_{z})}}.$$
Since $|\supp(\boldsymbol{w})| \geq n^{0.99}$, and since $\rho_{1,z}(\boldsymbol{w}) = O(v_{z}^{-2}\rho_{v_{z},\sum_{j=1}^{n}w_j z_{j}}(1))$ (since any ball in the complex plane of radius $1$ can be covered by $O(v_{z}^{-2})$ balls of radius $v_{z}$)  , the desired conclusion follows.
\end{proof}

\subsection{The Least Common Denominator}
The proof of \cref{thm:main} will be based on a `rounding argument' which extracts a `not-too-large' Gaussian integer vector certifying that the least singular value of a complex matrix is small (see \cite{jain2019combinatorial} for the most basic version of this argument). For this, 
we will use (albeit in a quite different manner from Rudelson and Vershynin) the notion of the Least Common Denominator (LCD) of a vector, and its connection to the L\'evy concentration function, as developed in \cite{rudelson2008littlewood}. 

\begin{remark}Our definition of the LCD is different from the ones appearing in the literature for the complex case, and has been made keeping in mind our application to rounding vectors.
\end{remark}
\begin{definition}[Least Common Denominator (LCD)]
Let $\boldsymbol{a}\in \C^{n}\setminus \{\boldsymbol{0}\}$. For $\gamma\in(0,1)$ and $\alpha>0$, define 
\[
\LCD_{\gamma,\alpha}(\boldsymbol{a}):=\inf_{\theta \in \C}\left\{ |\theta|>0:\dist(\theta \boldsymbol{a},(\Z+i\Z)^{n})<\min\{\gamma|\theta|\|\boldsymbol{a}\|_{2},\alpha\}\right\} .
\]
Note that the requirement that the distance is smaller than $\gamma|\theta|\| \boldsymbol{a}\|_{2}$
forces us to consider only non-trivial Gaussian integer points as approximations
of $\theta \boldsymbol{a}$. 
\end{definition}

The following theorem shows that vectors with large LCD have small L\'evy concentration function on scales which are larger (up to some small polynomial losses) than $\Omega(1/\text{LCD})$.

\begin{theorem}
\label{thm:LCD-controls-sbp}
Let $z$ denote a $C_z$-good complex random variable. Then, for every $\boldsymbol{a}\in \S^{2n-1}$, for every $\alpha \in (0,\sqrt{n}), \gamma \in (0,1)$, and for
$$\delta \geq \frac{n^{0.1}\alpha}{\LCD_{\alpha,\gamma}(\boldsymbol{a})},$$
we have
$$\rho_{\delta, z}(\boldsymbol{a}) \leq C_{\ref{thm:LCD-controls-sbp}}\left(\frac{\sqrt{n}\delta}{\gamma} + \exp\left(-C_{\ref{thm:LCD-controls-sbp}}^{-1}\alpha^{2}\right)\right),$$
where $C_{\ref{thm:LCD-controls-sbp}} \geq 1$ is a constant depending only on $C_z$.
\end{theorem}

A more precise version of this theorem appears for real random variables in \cite{rudelson2009smallest}. Actually, a version for complex random variables is also stated there although, as noted above, their definition of $\LCD$ is different from ours. The proof of \cref{thm:LCD-controls-sbp} follows from standard Fourier analytic arguments for the real case (in particular, we will use a crude version of the argument of Friedland and Sodin in \cite{friedland2007bounds}) once we use a novel `doubling trick'. 

We will need a preliminary result from Section 4 of \cite{tao2008random}.

\begin{definition}
Let $z$ be an arbitrary complex random variable. For any $w\in\C$, we define 
$$\|w\|_{z}^{2}:=\E\|\Re\{w(z_{1}-z_{2})\}\|_{\R/\Z}^{2},$$
where $z_{1},z_{2}$ denote i.i.d. copies of $z$ and $\|\cdot\|_{\R/\Z}$ denotes the distance to the nearest integer. 
\end{definition}

\begin{lemma}[Lemma 5.2 in \cite{tao2008random}]
\label{lemma:initial-fourier-bound}
Let
$\boldsymbol{v}:=(v_{1},\dots,v_{n})\in\C^{n}$ and let $z$ be an
arbitrary complex random variable. Then, 
$$\rho_{r,z}(\boldsymbol{v})\le e^{\pi r^{2}}P_{z}(\boldsymbol{v}) \leq  e^{\pi r^{2}}\int_{\C}\exp\left(-\sum_{i=1}^{n}\|v_{i}\xi\|_{z}^{2}/2-\pi|\xi|^{2}\right)d\xi.$$
Here, 
$$P_z(\boldsymbol{v}) := \E_{x_1,\dots,x_n}\exp(-\pi|v_1 x_1 + \dots + v_n x_n|^{2}),$$
where $x_1,\dots,x_n$ are i.i.d. copies of $(z_1 - z_2)\cdot \text{Ber} (1/2)$, with $z_1,z_2$  distributed as $z$, and  $\text{Ber}(1/2), z_1, z_2$ mutually independent.  
\end{lemma}

\begin{proof}[Proof of \cref{thm:LCD-controls-sbp}]
Since $\rho_{\delta,z}(\boldsymbol{a}) = \rho_{1,z}(\delta^{-1}\boldsymbol{a})$, it suffices to bound $\rho_{1,z}(\boldsymbol{v})$ for $\boldsymbol{v}:= \delta^{-1}\boldsymbol{a}$.
Let $\boldsymbol{w}\in \C^{2n}$ denote the vector whose first $n$ components are $\boldsymbol{v}$ and last $n$ components are $i\boldsymbol{v}$. Then, we have
  \begin{align*}
      \rho_{1,z}(\boldsymbol{v})^{2}
      &= \rho_{1,z}(\boldsymbol{v})\rho_{1,z}(i\boldsymbol{v})\\
      &\leq \exp(2\pi)P_{z}(\boldsymbol{v})P_{z}(i\boldsymbol{v})\\
      &\leq 2\exp(2\pi)P_{z}(\boldsymbol{w})\\
      &\leq 2\exp(2\pi)\int_{\C}\exp\left(-\sum_{j=1}^{n}\left(\|v_{j}\xi\|_{z}^{2} + \|i v_{j}\xi\|_{z}^{2}\right)/2 - \pi |\xi|^{2}\right)d\xi,
  \end{align*}
  where the first line uses $\rho_{1,z}(\boldsymbol{v}) = \rho_{1,z}(i\boldsymbol{v})$, the second line is due to \cref{lemma:initial-fourier-bound}, the third line follows from Lemma 4.5(iii) in \cite{tao2008random}, and the last line is again due to \cref{lemma:initial-fourier-bound}.
  
  Next, note that
  \begin{align*}
      \sum_{j=1}^{n}\left(\|v_j \xi\|_{z}^{2} + \|i v_j \xi\|_{z}^{2}\right)&=
      \E\sum_{j=1}^{n}\left(\|\Re\{v_j \xi (z_1 - z_2)\}\|_{\R/\Z}^{2} + \|\Re\{iv_j \xi (z_1-z_2)\}\|_{\R/\Z}^{2}\right)\\
      &= \E\sum_{j=1}^{n}\left(\|\Re\{v_j \xi (z_1 - z_2)\}\|_{\R/\Z}^{2} + \|\Im\{v_j \xi (z_1-z_2)\}\|_{\R/\Z}^{2}\right)\\
      &= \E\left[ \dist^{2}\left(\boldsymbol{v} \xi (z_1-z_2), (\Z + i\Z)^{n}\right)\right]\\
      &\geq \E\left[ \dist^{2}\left(\boldsymbol{v} \xi (z_1-z_2), (\Z + i\Z)^{n}\right) \bigg\vert |z_1 - z_2| \in [C_z^{-1}, C_{z}]\right]C_{z}^{-1},
  \end{align*}
  where the final inequality follows from the $C_z$-goodness of $z$.
  
  Therefore, from Jensen's inequality, we get that
  \begin{align}
      \rho_{1,z}(\boldsymbol{v})^{2}&
      \leq 2\exp(2\pi)\E \left[\int_{\C}\exp(-C_{z}^{-1}\dist^{2}\left(\boldsymbol{v} \xi (z_1-z_2), (\Z + i\Z)^{n}\right)/2 - \pi |\xi|^{2})d\xi \bigg\vert |z_1 - z_2| \in [C_{z}^{-1}, C_{z}] \right] \nonumber \\
      &\leq 2\exp(2\pi )\sup_{|y| \in [C_z^{-1}, C_{z}]}\int_{\C}\exp(-C_{z}^{-1}\dist^{2}\left(\boldsymbol{v} \xi y, (\Z + i\Z)^{n}\right)/2 - \pi |\xi|^{2})d\xi.
      \label{eqn:fourier-bound-appendix}
  \end{align}

Now, fix $y_0 \in \C$ with $|y_0| \in [C_z^{-1}, C_{z}]$; we will obtain a uniform (in $y_0$) upper bound on the integral appearing in \cref{eqn:fourier-bound-appendix}. Let
$$A:=\{\xi \in \C \mid \dist(\boldsymbol{v}\xi y_0, (\Z + i\Z)^{n}) \geq \alpha/2\}\cup \{\xi \in \C \mid |\xi| \geq \alpha\},$$
let $B:= \C \setminus A = B(0,\alpha)\setminus A$,
and split the integral above as
$$\int_{\C} = \int_{A} + \int_{B}.$$
Since 
$$\int_{A} \lesssim \exp\left(-\Omega_{C_{z}}({\alpha^{2}})\right),$$
it only remains to bound $\int_{B}$. 

For this, we begin by noting that if $\xi', \xi'' \in B$, then by the triangle inequality and the lattice structure of the Gaussian integers,
$$\dist\left(\boldsymbol{a}\delta^{-1}(\xi' - \xi'')y_0, (\Z+i\Z)^{n}\right) = \dist\left(\boldsymbol{v}(\xi' - \xi'')y_0, (\Z+i\Z)^{n}\right) < \alpha.$$
Hence, by the definition of $\LCD_{\gamma,\alpha}(\boldsymbol{a})$, we have one of two possibilities: either 
$$\delta^{-1}C_{z}|\xi' - \xi''| \geq \delta^{-1}|y_0||\xi' - \xi''| \geq \LCD_{\gamma, \alpha}(\boldsymbol{a})$$ or 
$$\gamma |\xi' - \xi''|\delta^{-1}C_{z}^{-1} \leq \gamma |\xi' - \xi''| \delta^{-1} |y_0| < \dist(\boldsymbol{v}(\xi' - \xi'')y_0, (\Z + i\Z)^{n}) < \sqrt{n}.$$
It follows that $B$ is contained in a union of balls of radius $C_z\sqrt{n}\delta/\gamma$ whose centers are separated by at least $\delta\LCD_{\gamma,\alpha}(\boldsymbol{a})/C_{z}$. Each such ball can contribute at most $\pi C_{z}^{2}n\delta^{2}/\gamma^{2}$ to the integral, and since $\delta \LCD_{\gamma,\alpha}(\boldsymbol{a}) \gg \alpha$, there is at most one such ball in $B$. It follows that
$$\int_{B} \leq \frac{\pi C_{z}^{2}n\delta^{2}}{\gamma^{2}}.$$
Finally, combining the estimates on $\int_{A}$ and $\int_{B}$ and using \cref{eqn:fourier-bound-appendix} completes the proof.
\end{proof}

\subsection{The counting problem in inverse Littlewood-Offord theory}
\label{sec:counting-problem}
The inverse Littlewood-Offord problem, posed by Tao and Vu \cite{tao2009inverse}, asks for the underlying reason that the L\'evy concentration function of a vector $\boldsymbol{v}\in \C^{n}$ can be large. Using deep Frieman-type results from additive combinatorics, they showed that, roughly speaking, the only reason for this to happen is that most of the coordinates of the vector $\boldsymbol{v}$ belong to a generalized arithmetic progression (GAP) of `small rank' and `small volume'. Their results \cite{tao2009inverse, tao2010sharp} were subsequently sharpened by Nguyen and Vu~\cite{nguyen2011optimal}, who proved an `optimal inverse Littlewood--Offord theorem'. We refer the reader to the survey \cite{nguyen2013small} and the textbook \cite{tao2006additive} for complete definitions and statements, and much more on both forward and inverse Littlewood-Offord theory. 

 Recently, motivated by applications, especially those in random matrix theory, the following \emph{counting variant} of the inverse Littlewood--Offord problem was isolated in work \cite{FJLS2018} of the author along with Ferber, Luh, and Samotij: for \emph{how many} vectors $\boldsymbol{a}$ in a given collection $\mathcal{A}\subseteq \Z^{n}$ is $\rho_{1,z}(\boldsymbol{a})$ greater than some prescribed value, where $z$ is a symmetric Bernoulli random variable? Indeed, the inverse Littlewood-Offord theorems are typically used precisely through such counting corollaries \cite{nguyen2013small}, and one of the main contributions of \cite{FJLS2018} (see Theorem 1.7 there) was to show that one may obtain useful bounds for the counting variant of the inverse Littlewood-Offord problem directly, \emph{without} providing a precise structural characterization like Tao-Vu. In fact, since this approach is not hampered by losses coming from the black-box application of various theorems from additive combinatorics, it provides quantitatively better bounds, and this was used in  \cite{FJLS2018, ferber2018singularity, jain2019combinatorial} to provide quantitative improvements for several problems in combinatorial random matrix theory. 
 
 
 
 A crucial ingredient in our proof will be the following extension of Theorem 1.7 of \cite{FJLS2018} due to the author 
 \begin{theorem}[Theorem 1.3 in \cite{jain2019c}]
\label{thm:counting-continuous}
Let $z$ be a $C_z$-good random variable.  For $\rho \in (0,1)$ (possibly depending on $n$), let
$$\boldsymbol{V}_{\rho} :=\left\{\boldsymbol{v}\in (\Z+i\Z)^{n}: \rho_{1,z}(\boldsymbol{v}) \geq \rho\right\}.$$ There exists a constant $C_{\ref{thm:counting-continuous}} \geq 1$, depending only on $C_z$, for which the following holds. Let $n,s,k\in \N$ with $1000C_{z} \leq k\leq \sqrt{s}\leq s \leq n/\log{n}$. If $\rho \geq C_{\ref{thm:counting-continuous}}\max\left\{ e^{-s/k},s^{-k/4}\right\}$ and $p$ is an odd prime such that $2^{n/s}\geq p \geq C_{\ref{thm:counting-continuous}}\rho^{-1}$, then 
$$\left|\varphi_p(\boldsymbol{V}_\rho)\right| \leq \left(\frac{5np^{2}}{s}\right)^{s} + \left(\frac{C_{\ref{thm:counting-continuous}}\rho^{-1}}{\sqrt{s/k}}\right)^{n},$$
where $\varphi_p$ denotes the natural map from $(\Z+i\Z)^{n}\to (\F_p+i\F_p)^{n}$.
\end{theorem}

\begin{remark}
The inverse Littlewood-Offord theorems may be used to deduce similar statements, \emph{provided we further assume that $\rho \geq n^{-C}$ for some constant $C>0$}. It is the freedom  of taking $\rho$ to be much smaller which allows us to obtain the exponential-type rate in \cref{thm:main}.
\end{remark}

\subsection{Norms of large projections of random matrices}
The key difficulty with extending the geometric techniques of Rudelson and Vershynin \cite{rudelson2008littlewood, rudelson2010non}) to the setting when the random variables have heavy tails is the lack of control on the operator norm of the random matrix. For our techniques, the following proposition will turn out to be an appropriate substitute for controlling the operator norm. 

For a subset $I\subseteq[n]$, let $P_I:\C^{n}\to \C^{n}$ denote the orthogonal projection onto the subspace spanned by the vectors $\{e_i:i\in I\}$. We have:
\begin{proposition}
\label{prop:operator-norm-control}
Let $N_n:=(m_{ij})$ be an $n\times n$ complex random matrix with i.i.d. entries, each with mean $0$ and variance $1$. For $\epsilon,\delta \in (0,1/2)$ with $\delta \geq 4\epsilon$, there exists $C_{\ref{prop:operator-norm-control}}(\epsilon) \geq 1$ such that, except with probability at most 
$C_{\ref{prop:operator-norm-control}}(\epsilon) \exp\left(-{n^{1-\epsilon}}/{8}\right)$, the following hold.  
\begin{enumerate}
    \item There exists $I \subseteq[n]$ with $|I| \geq n-2n^{1-\epsilon}$ such that 
    $$\|P_{I}N_{n}\|_{\infty \to 2} \leq C_{\ref{prop:operator-norm-control}}(1) n^{1+\epsilon}.$$
    \item For every $J\subseteq [n]$ with $|J|= n^{1-\delta}$, there exists some $I(J)\subseteq [n]$ such that $|I(J)|\geq n-2n^{1-\epsilon}$, and 
    $$\|P_{I(J)} N_n P_{J}\|_{\infty \to 2} \leq C_{\ref{prop:operator-norm-control}}(1)n^{1 + \epsilon - 0.5\delta}.$$
\end{enumerate}
\end{proposition}
\begin{remark}
A statement similar to the one above, and with some common proof ideas, already appears in the work of Rebrova and Vershynin \cite{rebrova2018norms}. In that work, the primary interest is in obtaining optimal bounds on the restricted operator norms and consequently, the proofs are much more involved. In contrast, we do not require such optimal bounds for our application, and are therefore able to give a much shorter proof of the above proposition.  
\end{remark}

The complete proof of this proposition is deferred to \cref{sec:proof-operator-norm}.

\section{Proof of \cref{thm:main}}
\label{sec:proof-main}
In this section, we will take $\alpha:= n^{1/100}$ and $\gamma := n^{-1/2}$. Moreover, since \cref{thm:main} is trivially true for $\eta \geq n^{-2}$, we will henceforth assume that $2^{-n^{0.0001}} \leq \eta < n^{-2}$. Recall that $M$ is a fixed $n\times n$ matrix with operator norm at most $n^{0.51}$; we set $M_n:= M+N_n$.\\

We decompose $\S^{2n-1}$ into $\Gamma^{1}(\eta) \cup \Gamma^{2}(\eta)$, where
$$\Gamma^{1}(\eta):= \left\{\boldsymbol{a}\in \S^{2n-1}: \LCD_{\alpha,\gamma}(\boldsymbol{a}) \geq n^{0.7}\cdot \eta^{-1} \right\} $$
and $\Gamma^{2}(\eta) := \S^{2n-1}\setminus \Gamma^{1}(\eta)$. Accordingly, we have
\begin{align*}
    \Pr\left(s_n(M_n)\leq \eta\right) \leq \Pr\left(\exists \boldsymbol{a}\in \Gamma^{1}(\eta): \|M_{n}\boldsymbol{a}\|_{2} \leq \eta\right) + \Pr\left(\exists \boldsymbol{a}\in \Gamma^{2}(\eta): \|M_{n}\boldsymbol{a}\|_{2} \leq \eta\right).
\end{align*}
Therefore, \cref{thm:main} follows from the following two propositions and the union bound. 
\begin{proposition}
\label{prop:eliminate-large-LCD}
$\Pr\left(\exists \boldsymbol{a} \in \Gamma^{1}(\eta): \|M_{n}\boldsymbol{a}\|_{2}\leq \eta \right) \leq 2nC_{\ref{thm:LCD-controls-sbp}}\left(n^{3/2}\eta + \exp(-C_{\ref{thm:LCD-controls-sbp}}^{-1}n^{1/50})\right)$. 
\end{proposition}
\begin{proposition}
\label{prop:eliminate-small-LCD}
$\Pr\left(\exists \boldsymbol{a} \in \Gamma^{2}(\eta): \|M_{n}\boldsymbol{a}\|_{2}\leq \eta \right) \leq C_{\ref{prop:eliminate-small-LCD}}\left(e^{-n^{0.98}} +  \exp(-c_{\ref{prop:eliminate-small-LCD}}n)\right),$
where $C_{\ref{prop:eliminate-small-LCD}}\geq 1$ and $c_{\ref{prop:eliminate-small-LCD}}>0$ are constants depending only on $z$.
\end{proposition}


The proof of \cref{prop:eliminate-large-LCD} is relatively simple, and follows from a conditioning argument developed in \cite{litvak2005smallest}, once we observe the crucial fact (\cref{thm:LCD-controls-sbp}) that for any $\boldsymbol{a}\in \Gamma^{1}(\eta)$, 
$$\rho_{\delta ,z}(\boldsymbol{a})  \lesssim \gamma^{-1}\sqrt{n}\delta + \exp(-\Omega(n^{1/50}))$$ for all $\delta \geq \eta$. We defer the (by now standard) details to \cref{sec:eliminate-large-LCD}.

The remainder of this section is devoted to the proof of \cref{prop:eliminate-small-LCD}.

\subsection{Reduction to Gaussian integer vectors}
Let $\mathcal{K}:= \{K\subseteq[n]: |K|\geq n-6n^{0.99}\}$ and
$$\boldsymbol{V} = \{\boldsymbol{v} \in (\Z+i\Z)^{n} \mid \|\Re(\boldsymbol{v})\|_{\infty}, \|\Im(\boldsymbol{v})\|_{\infty} \leq 2\eta^{-1}n^{3/4}\}.$$
As a first and crucial step towards the proof of \cref{prop:eliminate-small-LCD}, we will prove the following:
\begin{proposition}
\label{prop:reduction-to-integer-heavy-tailed}
With notation as above,
\begin{align*}
\Pr\left(\exists \boldsymbol{a} \in \Gamma^{2}(\eta):\|M_{n}\boldsymbol{a}\|_{2}\leq \eta \right) 
&\leq C_{\ref{prop:reduction-to-integer-heavy-tailed}}e^{-n^{0.99}/10} + \\ \Pr(\exists \boldsymbol{w} \in \boldsymbol{V}\text{ and } K\in \mathcal{K} &: \|P_{K}M_{n}\boldsymbol{w}\|_{2}\leq C_{\ref{prop:reduction-to-integer-heavy-tailed}}\min\{n^{0.21}\|\boldsymbol{w}\|_{2}, n^{0.711}\}),
\end{align*}
where $C_{\ref{prop:reduction-to-integer-heavy-tailed}}\geq 1$ is an absolute constant. 
\end{proposition}
\begin{remark}
As we will see shortly, the crucial point in the above proposition is that $n^{0.21} \ll n^{1/2 - \epsilon}$ and $n^{0.711} \ll n^{0.75 - \epsilon}$.
\end{remark}
\begin{proof}
Let $\epsilon = 0.01$, $\delta_{1} = 0.2$, and $\delta_{2} = 0.6$. Let $\mathcal{G}$ denote the event appearing in the conclusion of \cref{prop:operator-norm-control} for $(\epsilon,\delta_1)$ and $(\epsilon,\delta_2)$ simultaneously. Since $\Pr(\mathcal{G}^{c}) \leq 2C_{\ref{prop:operator-norm-control}}(0.01)\exp(-n^{0.99}/8)$, we may restrict ourselves to the event $\mathcal{G}$.  

Let $\boldsymbol{a}\in \Gamma^{2}(\eta)$. Then, by definition, there exists some $\theta \in \C$ with $0 < |\theta| \leq \LCD_{\alpha,\gamma}(\boldsymbol{a})\leq n^{3/4}\eta^{-1}$ and some $\boldsymbol{w}\in (\Z+i\Z)^{n}\setminus \{\boldsymbol{0}\}$ such that $\|\theta \boldsymbol{a} - \boldsymbol{w}\|_{2} \leq \min\{\gamma |\theta|,\alpha\}$. Note also that $\|\theta \boldsymbol{a}-\boldsymbol{w}\|_{\infty} \leq \min\{\gamma|\theta|,1\}$. To leverage the control we have over various norms associated to the matrix $M_n$, we decompose the `error' vector $\theta \boldsymbol{a} - \boldsymbol{w}$ into a `small' part (with respect to the $\ell^\infty$-norm), a `sparse and small' part, and a `very sparse' part. 

Accordingly, let $\boldsymbol{v}_{\text{sp}} \in \C^{n}$ denote the vector obtained by keeping the largest (in absolute value) $n^{0.4}$ coordinates of $\theta \boldsymbol{a} - \boldsymbol{w}$, let $\boldsymbol{v}_{\text{ss}}$ denote the vector obtained by keeping the next $n^{0.8}-n^{0.4}$ largest coordinates of $\theta \boldsymbol{a} - \boldsymbol{w}$, and let $\boldsymbol{v}_{\text{sm}} = \theta \boldsymbol{a} - \boldsymbol{w} - \boldsymbol{v}_{\text{sp}} - \boldsymbol{v}_{\text{ss}}$. Then, we have that $$\|\boldsymbol{v}_{\text{sp}}\|_{\infty}\leq \min\{\gamma|\theta|,1\},$$ and
\begin{align}
\label{eqn:bound-on-vsmall}
\|\boldsymbol{v}_{\text{ss}}\|_{\infty} \leq \frac{\min\{\gamma |\theta|,\alpha\}}{n^{0.2}} , \|\boldsymbol{v}_{\text{sm}}\|_{\infty} \leq \frac{\min\{\gamma |\theta|,\alpha\}}{n^{0.4}}.
\end{align}
Indeed, the first inequality is immediate from $\|\theta \boldsymbol{a} - \boldsymbol{w}\|_{\infty} \leq \min\{\gamma|\theta|,1\}$, whereas the second inequality follows from
$$\max\{n^{0.4}\|\boldsymbol{v}_{\text{ss}}\|_{\infty}^{2}, n^{0.8}\|\boldsymbol{v}_{\text{sm}}\|^{2}_{\infty}\}\leq \|\theta \boldsymbol{a} -\boldsymbol{w}\|^{2}_{2}.$$
Let $J_1\subseteq[n]$ denote the support of $\boldsymbol{v}_{\text{sp}}$ and let $J_{2}\subseteq[n]$ denote the support of $\boldsymbol{v}_{\text{sp}}+ \boldsymbol{v}_{\text{ss}}$. By extending these sets if need be, we may assume that $|J_1| = n^{0.4}$ and $|J_2| = n^{0.8}$. Moreover, since we have restricted to $N_n \in \mathcal{G}$, let $I\subseteq [n]$ denote a subset of size at least $n-2n^{1-\epsilon}$ with respect to which conclusion 1. of \cref{prop:operator-norm-control} holds. 

Note that since $\|M\| \leq n^{0.51}$, we have $\|MP_{J}\|_{\infty \to 2} \leq n^{0.51}\sqrt{|J|}$ for every $J\subseteq[n]$. Therefore,
$$\|P_{I}M_{n}\|_{\infty \to 2} \leq \|P_{I}N_{n}\|_{\infty \to 2} + \|P_{I}M_{n}\|_{\infty \to 2} \lesssim n^{1.01}$$
and similarly, 
\begin{align*}
\|P_{I(J_1)}M_{n}P_{J_1}\|_{\infty \to 2} 
& \lesssim n^{0.71},\\
\|P_{I(J_2)}M_{n}P_{J_2}\|_{\infty \to 2} 
& \lesssim n^{0.91}.
\end{align*}
Then, from the triangle inequality,  
we have
\begin{align*}
\|P_{I(J_1)}P_{I(J_2)}P_{I}M_{n}(\theta\boldsymbol{a}-\boldsymbol{w})\|_{2} & \leq\|P_{I}M_{n}\boldsymbol{v}_{\text{sm}}\|_{2}+\|P_{I(J_2)}M_{n}\boldsymbol{v}_{\text{ss}}\|_{2} + \|P_{I(J_1)}M_{n}\boldsymbol{v}_{\text{sp}}\|_{2}\\
&= \|P_{I}M_{n}\boldsymbol{v}_{\text{sm}}\|_{2}+\|P_{I(J_2)}M_{n}P_{J_2}\boldsymbol{v}_{\text{ss}}\|_{2} + \|P_{I(J_1)}M_{n}P_{J_1}\boldsymbol{v}_{\text{sp}}\|_{2}\\
 \leq\|P_{I}M_{n}\|_{\infty\to2}\|\boldsymbol{v}_{\text{sm}}\|_{\infty}&+\|P_{I(J_2)}M_{n}P_{J_2}\|_{\infty\to2}\|\boldsymbol{v}_{\text{ss}}\|_{\infty} + \|P_{I(J_1)}M_{n}P_{J_1}\|_{\infty\to2}\|\boldsymbol{v}_{\text{sp}}\|_{\infty}\\
 & \lesssim \left( n^{0.61}{\min\{\gamma|\theta|,\alpha\}}+ n^{0.71}{\min\{\gamma|\theta|,\alpha\}} +  n^{0.71}\min\{\gamma |\theta|, 1\}\right)\\
 & \lesssim \left(\min\{n^{0.21}|\theta|, n^{0.711 }\} + \min\{n^{0.21}|\theta|,n^{0.71}\}\right)\\
 &\lesssim \min\{n^{0.21}|\theta|, n^{0.711}\},
\end{align*}
where the second line uses that $P_{J_2}\boldsymbol{v}_{\text{ss}} = \boldsymbol{v}_{\text{ss}}$ and $P_{J_1}\boldsymbol{v}_{\text{sp}} = \boldsymbol{v}_{\text{sp}}$ ; the fourth line uses the above estimates on the $\infty$-to-$2$ norms and \cref{eqn:bound-on-vsmall}, and the fifth line uses the parameter value $\gamma = n^{-1/2}$. 

Thus, if $\|M_n\boldsymbol{a}\|_{2} \leq \eta$, it follows from the triangle inequality that 
\begin{align*}
\|P_{I(J_1)}P_{I(J_2)}P_IM_{n}\boldsymbol{w}\|_{2} & =\|P_{I(J_1)}P_{I(J_2)}P_IM_{n}(\boldsymbol{w}-\theta \boldsymbol{a})+P_{I(J_1)}P_{I(J_2)}P_IM_{n}(\theta \boldsymbol{a})\|_{2}\\
 & \leq\|P_{I(J_1)}P_{I(J_2)}P_IM_{n}(\theta \boldsymbol{a}-\boldsymbol{w})\|_{2}+|\theta|\cdot \|M_{n}\boldsymbol{a}\|_{2}\\
 & \lesssim \min\{n^{0.21}|\theta|, n^{0.711}\}+|\theta|\eta\\
 &\lesssim \min\{n^{0.21}|\theta|, n^{0.711}\}\\
 &\lesssim \min\{n^{0.21}\|\boldsymbol{w}\|_{2},n^{0.711}\},
\end{align*}
where the fourth line follows since $\eta \ll n^{0.21}$ and $|\theta| \eta \leq n^{0.7} \ll n^{0.711}$, and the last line follows since $\|\boldsymbol{w}\|_{2} \geq |\theta|(1-\gamma) \geq |\theta|/2$. Since $|I(J_1)^{c} \cup I(J_2)^{c} \cup I^{c}| \leq |I(J_1)^{c}| + |I(J_2)^{c}| + |I^{c}| \leq 6n^{0.99}$, we get the desired conclusion. 
\end{proof}

In view of \cref{prop:reduction-to-integer-heavy-tailed}, it suffices to show the following in order to prove \cref{prop:eliminate-small-LCD}, and hence, complete the proof of \cref{thm:main}.
\begin{proposition}
\label{prop:integers-heavy-tailed}
$\Pr(\exists \boldsymbol{w} \in \boldsymbol{V}\text{ and } K\in \mathcal{K} : \|P_{K}M_{n}\boldsymbol{w}\|_{2}\leq  C_{\ref{prop:reduction-to-integer-heavy-tailed}}\min\{n^{0.21}\|\boldsymbol{w}\|_{2}, n^{0.711}\}) \leq C_{\ref{prop:integers-heavy-tailed}} e^{-c_{\ref{prop:integers-heavy-tailed}}n},$
where $C_{\ref{prop:integers-heavy-tailed}}\geq 1$ and $c_{\ref{prop:integers-heavy-tailed}} > 0$ are constants depending only on $z$.
\end{proposition}
The proof of this proposition is the content of the next two subsections. 
\subsection{Dealing with sparse Gaussian integer vectors}
Throughout this subsection and the next one, $p = 2^{n^{0.001}}$ is a prime. Note, in particular, that $p \gg \eta^{-1}n^{3/4}$. 
The proof of \cref{prop:integers-heavy-tailed} proceeds in two steps. The first step is to show that the probability of the event appearing in \cref{prop:integers-heavy-tailed} is small, provided we restrict ourselves only to sufficiently sparse Gaussian integer vectors. 
Let
$$\boldsymbol{S}:= \{\boldsymbol{w}\in (\Z+i\Z)^{n}\setminus \{\boldsymbol{0}\} \mid \|\Re(\boldsymbol{w})\|_{\infty}, \|\Im(\boldsymbol{w})\|_{\infty} \leq p,  |\supp(\boldsymbol{w})| \leq n^{0.99}\}.$$
\begin{lemma}
\label{lemma:sparse-vectors-heavy-tails}
$\Pr\left(\exists \boldsymbol{w} \in \boldsymbol{S}\text{ and }K\in \mathcal{K}: \|P_K M_{n}\boldsymbol{w}\|_{2}  \leq C_{\ref{prop:reduction-to-integer-heavy-tailed}}n^{0.21}\|\boldsymbol{w}\|_{2}\right) \leq C_{\ref{lemma:sparse-vectors-heavy-tails}}  \exp(-c_{\ref{lemma:invertibility-single-vector}}n/4),$ where $C_{\ref{lemma:sparse-vectors-heavy-tails}} \geq 1$ is an absolute constant.
\end{lemma}
\begin{proof}
By taking the union bound over all the at most $n\binom{n}{6n^{0.99}} \ll \exp(n^{0.991})$ choices of $K \in \mathcal{K}$, it suffices to show that for a fixed $K_0\in \mathcal{K}$,
$$\Pr\left(\exists \boldsymbol{w} \in \boldsymbol{S}: \|P_{K_0} M_{n}\boldsymbol{w}\|_{2}  \leq C_{\ref{prop:reduction-to-integer-heavy-tailed}}n^{0.21}\|\boldsymbol{w}\|_{2}\right) \leq C  \exp(-c_{\ref{lemma:invertibility-single-vector}}n/2)$$
for some absolute constant $C\geq 1$.
The number of vectors $\boldsymbol{w} \in \boldsymbol{S}$ is at most
$$\binom{n}{n^{0.99}}(3p^{2})^{n^{0.99}} \ll 2^{n^{0.992}}.$$
By \cref{lemma:invertibility-single-vector} applied to the matrix $P_{K_0}M_n$, for any such vector, 
$$\Pr\left(\|P_{K_0} M_{n}\boldsymbol{w}\|_{2} \leq c_{\ref{lemma:invertibility-single-vector}}\sqrt{n}\|\boldsymbol{w}\|_{2}/2\right) \leq  \exp(-c_{\ref{lemma:invertibility-single-vector}}n).$$
Therefore, the union bound gives the desired conclusion. 
\end{proof}

\subsection{Dealing with non-sparse Gaussian integer vectors}
It remains to deal with Gaussian integer vectors with support of size at least $n^{0.99}$. Let
$$\boldsymbol{W}:= \left\{\boldsymbol{w}\in (\Z+i\Z)^{n}\setminus \{\boldsymbol{0}\} \mid \|\Re(\boldsymbol{w})\|_{\infty}, \|\Im(\boldsymbol{w})\|_{\infty} \leq \eta^{-4},  |\supp(\boldsymbol{w})| \geq n^{0.99} \right\}.$$
Note that for our choice of parameters, the natural map
$$\varphi_{p}: \boldsymbol{W} \to (\F_{p}+i\F_p)^{n}$$
is injective.

In view of \cref{lemma:sparse-vectors-heavy-tails}, since $\eta \leq n^{-2}$, and taking the union bound over all the at most $n\binom{n}{6n^{0.99}} \ll \exp(n^{0.991})$ choices of $K\in \mathcal{K}$, the following proposition suffices to prove \cref{prop:integers-heavy-tailed}.

\begin{proposition}
\label{prop:non-sparse-vectors-heavy-tails}
For all $K_0 \in \mathcal{K}$, $$\Pr\left(\exists \boldsymbol{w} \in \boldsymbol{W}: \|P_{K_0} M_{n}\boldsymbol{w}\|_{2} \leq C_{\ref{prop:reduction-to-integer-heavy-tailed}}n^{0.711} \right) \leq C_{\ref{prop:non-sparse-vectors-heavy-tails}} \exp(-c_{\ref{prop:non-sparse-vectors-heavy-tails}}n),$$
where $C_{\ref{prop:non-sparse-vectors-heavy-tails}}\geq 1$ and $c_{\ref{prop:non-sparse-vectors-heavy-tails}} > 0$ are constants depending only on $z$. 
\end{proposition}
The proof of \cref{prop:non-sparse-vectors-heavy-tails} is accomplished by a simple union bound. To execute this, we need the following preliminary claims.
\begin{claim}
\label{claim:lower-bound-anticonc-W}
For all $\boldsymbol{w}\in \boldsymbol{W}$,
$\rho_{1,z}(\boldsymbol{w}) \geq n^{-1/2}\eta^{4}/10$.
\end{claim}
\begin{proof}
The random variable $\sum_{j=1}^{n}w_j \xi_j$ has mean $0$ and variance at most $ n\eta^{-8}$. Therefore, by Markov's inequality, 
$$\Pr\left(\left|\sum_{j=1}^{n}w_j \xi_j\right| \leq 2\sqrt{n}\eta^{-4} \right) \geq \frac{3}{4}.$$
Hence, by the pigeonhole principle, it follows that 
$$\rho_{1,\xi}(\boldsymbol{w}) \geq n^{-1/2}\eta^{4}/10,$$
as desired. 
\end{proof}

For the next claim, let
$$\boldsymbol{W}_{t}:= \{\boldsymbol{w}\in \boldsymbol{W}: \rho_{1,\xi}(\boldsymbol{w})\in [t,2t)\}.$$
Note that the previous claim along with \cref{claim:anticonc-W} shows that $\boldsymbol{W}_{t}$ is nonempty only if $n^{-1/2}\eta^{4}/10\leq t\leq C_{\ref{claim:anticonc-W}}n^{-0.495}$.

\begin{claim}
\label{claim:size-of-W_t}
For all $n^{-1/2}\eta^{4}/10\leq t\leq C_{\ref{claim:anticonc-W}}n^{-0.495}$, 
$$|\boldsymbol{W}_{t}| \leq C_{\ref{claim:size-of-W_t}}\left(\frac{C_{\ref{thm:counting-continuous}}t^{-1}}{n^{0.45}}\right)^{n},$$
where $C_{\ref{claim:size-of-W_t}}\geq 1$ is a constant depending only on $C_{\ref{thm:counting-continuous}}, C_{\ref{claim:anticonc-W}}$. 
\end{claim}
\begin{proof}
Fix $s= n^{0.997}$ and $k= n^{0.097}$. Then, $1\ll k\leq \sqrt{s} \leq s \leq n/\log{n}$, $n^{-1/2}\eta^{4} \gg \max\{e^{-s/k}, s^{-k/4}\}$, and $2^{n/s} \geq p \gg n^{1/2}\eta^{-4}$. Hence, for large enough $n$, the hypotheses of \cref{thm:counting-continuous} are satisfied, so that 
\begin{align*}
    |\boldsymbol{W}_t| 
    &= |\varphi_p(\boldsymbol{W}_t)|\\
    &\leq |\varphi_p(\boldsymbol{V}_t)|\\
    &\leq \left(\frac{5np^{2}}{s}\right)^{s} + \left(\frac{C_{\ref{thm:counting-continuous}}t^{-1}}{n^{0.45}}\right)^{n}\\
    &\leq 2\left(\frac{C_{\ref{thm:counting-continuous}}t^{-1}}{n^{0.45}}\right)^{n},
\end{align*}
where the first line follows from the injectivity of $\varphi_{p}$ on $\boldsymbol{W}$, the third line follows from \cref{thm:counting-continuous}, and the last line follows since $t^{-1} \gg n^{0.49}$. 
\end{proof}

We now have all the ingredients to prove \cref{prop:non-sparse-vectors-heavy-tails}.

\begin{proof}[Proof of \cref{prop:non-sparse-vectors-heavy-tails}]
Let $D_{x}$ denote the unit polydisc in $\C^{n}$ centered at $x$. 
For all $n$ sufficiently large, we have
\begin{align*}
    \Pr\left(\exists \boldsymbol{w}\in \boldsymbol{W}: \|P_{K_0}M_n \boldsymbol{w}\|_{2} \leq C_{\ref{prop:reduction-to-integer-heavy-tailed}}n^{0.711}\right)
\leq& \sum_{t=0.1n^{-1/2}\eta^{4}}^{n^{-0.494}}\Pr\left(\exists \boldsymbol{w}\in \boldsymbol{W}_{t}: \|P_{K_0}M_n\boldsymbol{w}\|_{2} \leq C_{\ref{prop:reduction-to-integer-heavy-tailed}}n^{0.711}\right)\\
\lesssim \sum_{t=0.1n^{-1/2}\eta^{4}}^{n^{-0.494}}&\sum_{x\in B(0,n^{0.712})\cap (\Z + i\Z)^{n}}\Pr\left(\exists \boldsymbol{w}\in \boldsymbol{W}_{t}: P_{K_0}M_{n}\boldsymbol{w} \in D_{x}\right)\\
\leq \sum_{t=0.1n^{-1/2}\eta^{4}}^{n^{-0.494}}&(400n^{0.212})^{2n}\sup_{x\in (\Z+i\Z)^{n}}\Pr\left(\exists \boldsymbol{w}\in \boldsymbol{W}_{t}: P_{K_0}M_n\boldsymbol{w}\in D_{\boldsymbol{x}}\right)\\
\leq \sum_{t=0.1n^{-1/2}\eta^{4}}^{n^{-0.494}}&(16000n^{0.424})^{n}|\boldsymbol{W}_t|(2t)^{|K_0|}\\
\lesssim \sum_{t=0.1n^{-1/2}\eta^{4}}^{n^{-0.494}}&(16000n^{0.424})^{n}\left(\frac{C_{\ref{thm:counting-continuous}}t^{-1}}{n^{0.45}}\right)^{n}\cdot(2t)^{n-6n^{0.99}}\\
\lesssim \sum_{t=0.1n^{-1/2}\eta^{4}}^{n^{-0.494}}&(32000C_{\ref{thm:counting-continuous}}n^{-0.02})^{n}\cdot t^{-6n^{0.99}}\\
\lesssim& n\cdot(32000C_{\ref{thm:counting-continuous}}n^{-0.02})^{n}\cdot {\eta}^{-30n^{0.99}}\\
\lesssim& n\cdot(32000C_{\ref{thm:counting-continuous}}n^{-0.02})^{n}\cdot 2^{30n^{0.991}}\\
\leq& C_{\ref{prop:non-sparse-vectors-heavy-tails}}\exp(-c_{\ref{prop:non-sparse-vectors-heavy-tails}}n) , 
\end{align*}
where the third line follows since the number of points of $(\Z+i\Z)^{n}$ in $B(0,n^{0.712})$ is at most $(400n^{0.212})^{2n}$, the fifth line follows from \cref{claim:size-of-W_t}, and the seventh and eighth lines follow from the assumed bounds on $\eta$. 
\end{proof}

\bibliographystyle{abbrv}
\bibliography{least-singular-value}

\appendix

\section{Proof of \cref{prop:operator-norm-control}}
\label{sec:proof-operator-norm}
The proof will make use of the subgaussian concentration inequality, which we now recall. 
\begin{definition}
\label{defn:subgaussian}
A random variable $X$ is said to be $C$-subgaussian if, for all $t>0$,
$$\Pr\left(|X|>t\right) \leq 4\exp\left(-\frac{t^2}{C^2}\right).$$
\end{definition}

\begin{lemma}[see, e.g., Corollary 5.17 in \cite{vershynin2010introduction}]
\label{lemma:subgaussian-concentration}
There exists an absolute constant $C_{\ref{lemma:subgaussian-concentration}}>0$ with the following property. Let $X_1,\dots,X_n$ be independent centered $\tilde{C}_{\xi}$-subgaussian random variables. Then,
$$\Pr\left(\sum_{i=1}^{n}|X_i|^{2} \geq C_{\ref{lemma:subgaussian-concentration}}\tilde{C}_{\xi}^{2}n\right) \leq \exp(-2n).$$
\end{lemma}

We begin with a simple lemma showing that, with high probability, most rows of a random matrix with i.i.d. centered entries of finite variance have small $\ell_{1}$ and $\ell_{2}$ norms. 
\begin{lemma}
\label{lemma:proj_control_basic}
Let $A:=(a_{ij})$ be an $n\times m$ complex random matrix with i.i.d. entries,
each with mean $0$ and variance $1$. For $\epsilon\in(0,1/2)$, let $I\subseteq[n]$ denote the (random) subset of coordinates
such that for each $i\in I$,
\[
\left(\sum_{j=1}^{m}|a_{ij}|^{2}\leq n^{2\epsilon}m\right)\bigwedge\left(\left|\sum_{j=1}^{m}a_{ij}\right|\leq n^{\epsilon}\sqrt{m}\right).
\]
Then, 
\[
\Pr\left(|I^{c}|\geq2n^{1-\epsilon}\right)\leq2\exp\left(-\frac{n^{1-\epsilon}}{4}\right).
\]
\end{lemma}
\begin{proof}
Since for each $i\in[n]$,
\[
\E\left[\left|\sum_{j=1}^{m}a_{ij}\right|^{2}\right]=\E\left[\sum_{j=1}^{m}|a_{ij}|^{2}\right]=m,
\]
it follows from Markov's inequality that 
\[
\Pr\left(\sum_{j=1}^{m}|a_{ij}|^{2}>n^{2\epsilon}m\right)\leq n^{-2\epsilon}
\]
and 
\[
\Pr\left(\left|\sum_{j=1}^{m}a_{ij}\right|>n^{\epsilon}\sqrt{m}\right)\leq n^{-2\epsilon}.
\]
Let $I_{1}\subseteq[n]$ denote the subset of coordinates such that
for each $i\in I_{1}$, 
\[
\sum_{j=1}^{m}|a_{ij}|^{2}\leq n^{2\epsilon}m
\]
and let $I_{2}\subseteq[n]$ denote the subset of coordinates such
that for each $i\in I_{2}$, 
\[
\left|\sum_{j=1}^{m}a_{ij}\right|\leq n^{\epsilon}\sqrt{m}.
\]
Since the rows of the matrix are independent, it follows from the
standard Chernoff bound that for $k\in\{1,2\}$
\[
\Pr\left(|I_{k}^{c}|\geq n^{1-\epsilon}\right)\leq\exp\left(-\frac{n^{1-\epsilon}}{4}\right).
\]
Hence, by the union bound,
\[
|I^{c}|\leq|I_{1}^{c}|+|I_{2}^{c}|\leq2n^{1-\epsilon}, \] 
except with probability at most $2\exp\left(-\frac{n^{1-\epsilon}}{4}\right).$
\end{proof}
The next proposition controls the $\infty \to 2$ operator norm of a random matrix with i.i.d. entries, conditioned on no row having $\ell_1$ or $\ell_2$ norm which is `too large', and essentially appears as Proposition 3.10 in \cite{rebrova2018coverings}. Since our statement uses somewhat different parameters than in \cite{rebrova2018coverings}, we provide a complete proof below for the reader's convenience.   
\begin{proposition}
\label{prop:control-infty-to-2-norm}
Fix $\epsilon \in (0,1/2)$. Let $B:=(b_{ij})$ be a fixed $n\times m$ complex matrix, with $0.9n \leq m \leq 1.1n$, such that the $\ell_2$
norm of every row is at most $n^{\epsilon}\sqrt{m}$ and such that for all $i\in[n]$,
\[
\left|\sum_{j=1}^{m}b_{ij}\right|\leq n^{\epsilon}\sqrt{m}.
\]
Let $\pi_{1},\dots,\pi_{n}$ be independent random permutations uniformly
distributed on the symmetric group $S_{m}$, and let $\tilde{B}:=(\tilde{b}_{ij})$ denote
the random $n\times m$ complex matrix whose entries are given by 
\[
\tilde{b}_{ij}:=b_{i,\pi_{i}(j)}.
\]
Then, 
\[
\Pr\left(\|\tilde{B}\|_{\infty\to2}\geq C_{\ref{prop:control-infty-to-2-norm}} \sqrt{mn} n^{\epsilon}\right)\leq \exp(-2n),
\]
where $C_{\ref{prop:control-infty-to-2-norm}} \geq 1$ is an absolute constant. 
\end{proposition}
The following concentration inequality will be used to establish the subgaussianity of certain random variables appearing in the proof of \cref{prop:control-infty-to-2-norm}. It appears as Lemma 3.9 in \cite{rebrova2018coverings}, and is a direct application of Theorem 7.8 in \cite{milman2009asymptotic}. 
\begin{lemma}[Lemma 3.9 in \cite{rebrova2018coverings}]
\label{lemma:talagrand-corollary}
Let $\boldsymbol{y}:=(y_{1},\dots,y_{m})$ be a non-zero complex
vector and let $\boldsymbol{v}\in\{\pm1\}^{m}$. Consider the function
$f:S_{m}\to\C$ defined by 
\[
f(\pi):=\sum_{j=1}^{m}v_{\pi(j)}y_{j}.
\]
Then, for all $t>0$,
\[
\Pr\left(\left|f(\pi)-\E f\right|\geq t\right)\leq4\exp\left(-\frac{t^{2}}{128\|\boldsymbol{y}\|_{2}^{2}}\right).
\]
\end{lemma}
\begin{remark}
In \cite{rebrova2018coverings}, the above lemma is stated and proved (with better constants) for real vectors $\boldsymbol{y}$. However, the version above for complex vectors immediately follows from this by separately considering the real and imaginary parts of $f$ and using the union bound. 
\end{remark}
\begin{proof}[Proof of \cref{prop:control-infty-to-2-norm}]
If $\|\tilde{B}\|_{\infty\to 2} \geq C_{\ref{prop:control-infty-to-2-norm}}\sqrt{mn}n^{\epsilon}$, then there exists a complex vector $\boldsymbol{w} = \boldsymbol{w_1} + i \boldsymbol{w_2}$, where $\boldsymbol{w_1},\boldsymbol{w_2}\in \R^{m}$ and  $\|\boldsymbol{w_1}\|_{\infty}, \|\boldsymbol{w_2}\|_{\infty}\leq 1$, such that 
$$\|\tilde{B}\boldsymbol{w}_{1}\|_{2} + \|\tilde{B}\boldsymbol{w}_{2}\|_{2}\geq \|\tilde{B}\boldsymbol{w}\|_{2} \geq C_{\ref{prop:control-infty-to-2-norm}}\sqrt{mn}n^{\epsilon}.$$ 
Therefore, it suffices to control the $\infty$-to-$2$ norm of $\tilde{B}$ restricted to vectors in $\R^{m}$.
For this, it suffices by convexity and the union bound to show that
for any fixed $\boldsymbol{v}\in\{\pm1\}^{m}$, 
\[
\Pr\left(\|\tilde{B}\boldsymbol{v}\|^{2}_{2}\geq (128C_{\ref{lemma:subgaussian-concentration}}+2)mn^{1+2\epsilon}\right)\leq\exp(-2n-m\ln2).
\]
To see this, we begin by noting that the random variables $X_{i}:=\langle\tilde{B}\boldsymbol{v},e_{i}\rangle$
are independent and 
\[
X_{i}\sim\sum_{j=1}^{m}v_{\pi_{i}(j)}b_{ij}.
\]
In particular, if $\ell$ denotes the number of ones in $(v_{1},\dots,v_{m})$,
then 
\[
\left|\E[X_{i}]\right|=\left|\sum_{j=1}^{m}\E\left[v_{\pi_{i}(j)}\right]b_{ij}\right|=\left|\sum_{j=1}^{m}\frac{2\ell-m}{m}b_{ij}\right|\leq\left|\sum_{j=1}^{m}b_{ij}\right|\leq n^{\epsilon}\sqrt{m}.
\]
By \cref{lemma:talagrand-corollary}, for all $t>0$, we have
\[
\Pr\left(\left|X_{i}-\E[X_{i}]\right|\geq t\right)\leq4\exp\left(-\frac{t^{2}}{128\|\boldsymbol{b_{i}}\|_{2}^{2}}\right)\leq4\exp\left(-\frac{t^{2}}{128mn^{2\epsilon}}\right).
\]
In particular, the random variables $n^{-\epsilon}m^{-1/2}|X_{i}-\E[X_{i}]|$
are $16$-subgaussian so that by \cref{lemma:subgaussian-concentration}
\[
\Pr\left(\sum_{i=1}^{n}|X_{i}-\E[X_{i}]|^{2}\geq 256C_{\ref{lemma:subgaussian-concentration}}mn^{1+2\epsilon}\right)\leq\exp\left(-4n\right)\leq \exp\left(-2n-m\ln{2}\right).
\]
Finally, since  
\begin{align*}
\sum_{i=1}^{n}|X_{i}|^{2} & =\sum_{i=1}^{n}\left|X_{i}-\E[X_{i}]+\E[X_{i}]\right|^{2}\\
 & \leq2\sum_{i=1}^{n}\left|X_{i}-\E[X_{i}]\right|^{2}+2\sum_{i=1}^{n}|\E[X_{i}]|^{2}\\
 & \leq2\sum_{i=1}^{n}\left|X_{i}-\E[X_{i}]\right|^{2}+2mn^{1+2\epsilon},
\end{align*}
it follows that 
\begin{align*}
\Pr\left(\sum_{i=1}^{n}|X_{i}|^{2}\geq (256C_{\ref{lemma:subgaussian-concentration}}+2)mn^{1+2\epsilon}\right)&\leq\Pr\left(\sum_{i=1}^{n}\left|X_{i}-\E[X_{i}]\right|^{2}\geq 256C_{\ref{lemma:subgaussian-concentration}}mn^{1+2\epsilon}\right)\\
&\leq\exp\left(-2n-m\ln2\right),
\end{align*}
which completes the proof. 
\end{proof}

Given the above results, \cref{prop:operator-norm-control} is almost immediate.
\begin{proof}[Proof of \cref{prop:operator-norm-control}] \textbf{1.}
Let $N_n$ be the $n\times n$ complex random matrix appearing in the statement of the proposition, and let $\mathcal{E}$ denote the `good' event appearing in \cref{lemma:proj_control_basic} i.e. $\mathcal{E}$ is the event that there exists some $I\subseteq[n]$ with $|I|\geq n-2n^{1-\epsilon}$ such that for all $i\in I$,  
\[
\left(\sum_{j=1}^{n}|m_{ij}|^{2}\leq n^{1+2\epsilon}\right)\bigwedge\left(\left|\sum_{j=1}^{n}m_{ij}\right|\leq n^{(1/2)+\epsilon}\right).
\]
Since $\Pr(\mathcal{E}^c) \leq 2\exp(-n^{1-\epsilon}/4)$ by \cref{lemma:proj_control_basic}, it suffices to show that
$$\Pr\left(\left\{\inf_{i\in \mathcal{I}}\|P_{I}N_{n}\|_{\infty \to 2} \geq C_{\ref{prop:control-infty-to-2-norm}} n^{1+\epsilon} \right\}\cap \mathcal{E}\right) \leq \exp(-n),$$
where $\mathcal{I}$ denotes the collection of subsets of $[n]$ of size at least $n-2n^{1-\epsilon}$. For this, note that since both the event $\mathcal{E}$ as well as our distribution on $n\times n$ matrices are invariant under permuting each row of $N_n$ separately, it suffices to show the following: for each (fixed) $n\times n$ complex matrix $A_n$ for which there exists a subset $I\subseteq [n]$ as above,
$$\Pr\left(\|P_{I} \tilde{A}_{n}\|_{\infty \to 2} \geq C_{\ref{prop:control-infty-to-2-norm}} n^{1+\epsilon} \right) \leq \exp(-n),$$
where $\tilde{A}_{n}$ is the random complex matrix obtained by permuting each row of $A_n$ independently and uniformly. But this follows immediately from \cref{prop:control-infty-to-2-norm} applied to the $n\times n$ matrix $P_I A_n$. \\

\textbf{2.} The proof of this part is very similar to the previous one. Let $\mathcal{J}$ denote the collection of all subsets of $[n]$ of size $n^{1-\delta}$ and let $\mathcal{I}$ denote the collection of all subsets of $[n]$ of size at least $n-2n^{1-\epsilon}$. We show that the desired conclusion in 2. holds with sufficiently high probability for fixed $J\in \mathcal{J}$; the proof is completed by taking the union bound over the at most $$\binom{n}{n^{1-\delta}} \leq \exp(n^{1-\delta}\log{n}) \leq C(\epsilon)\exp(n^{1-3\epsilon})$$ choices for $J\in \mathcal{J}$, where $C(\epsilon)\geq 1$ depends only on $\epsilon$, and the last inequality uses that $\delta \geq 4\epsilon$.

For such a fixed $J\in \mathcal{J}$, let $\mathcal{E}_{\epsilon,\delta}$ denote the event that there exists some $I\in \mathcal{I}$ such that for all $i\in I$, 
\[
\left(\sum_{j\in J}|m_{ij}|^{2}\leq n^{2\epsilon}|J|\right)\bigwedge\left(\left|\sum_{j\in J}m_{ij}\right|\leq n^{\epsilon}\sqrt{|J|}\right).
\]
  As before, by \cref{lemma:proj_control_basic} applied to the operator $N_n P_J$ viewed as an $n\times |J|$ matrix, we see that $\Pr(\mathcal{E}_{\epsilon,\delta}^{c}) \leq 2\exp(-n^{1-\epsilon}/4)$. Therefore, it suffices to show that   
$$\Pr\left(\left\{\inf_{i\in \mathcal{I}}\|P_{I}N_{n}P_J\|_{\infty \to 2} \geq C_{\ref{prop:control-infty-to-2-norm}} n^{1+\epsilon-0.5\delta} \right\}\cap \mathcal{E}_{\epsilon,\delta}\right) \leq \exp(-n).$$
But this follows by exactly the same argument (using \cref{prop:control-infty-to-2-norm}) as above. 
\end{proof} 
 
\section{Proof of \cref{prop:eliminate-large-LCD}}
\label{sec:eliminate-large-LCD}
\begin{proof}[Proof of \cref{prop:eliminate-large-LCD} following \cite{litvak2005smallest, tao2009inverse}]
Since $M_n^{\dagger}$ and $M_{n}$ have the same singular values, it follows that a necessary condition for a matrix $M_n$ to satisfy the event in \cref{prop:eliminate-large-LCD} is that there exists a unit row vector $\boldsymbol{a'}=(a'_{1},\dots,a'_{n})$
such that $\|\boldsymbol{a'}^{T}M_{n}\|_{2}\leq \eta$. To every matrix $M_n$, associate such a vector $\boldsymbol{a'}$ arbitrarily (if one exists) and denote it by $\boldsymbol{a'}_{M_n}$; this leads to a partition of the space of all matrices with least singular value at most $\eta$. Then, by taking a union bound, it suffices to show the following. 
\begin{align}
\label{eqn:intersected-event}
\Pr\left(\exists \boldsymbol{a}\in \Gamma^{1}(\eta): \|M_n \boldsymbol{a}\|_{2} \leq \eta 
\bigwedge \|\boldsymbol{a'}_{M_n}\|_{\infty} = |a'_n| \right) \leq 2C_{\ref{thm:LCD-controls-sbp}}\left(n^{3/2}\eta + \exp(-C_{\ref{thm:LCD-controls-sbp}}^{-1}\sqrt{n})\right).
\end{align}
To this end, we expose the first $n-1$ rows $X_{1},\dots,X_{n-1}$ of $M_{n}$. Note that if there is some $\boldsymbol{a}\in\Gamma^{1}(\eta)$ satisfying $\|M_{n}\boldsymbol{a}\|_{2}\leq \eta$,
then there must exist a vector $\boldsymbol{y}\in \Gamma^{1}(\eta)$, depending only on
the first $n-1$ rows $X_{1},\dots,X_{n-1}$, such that 
\[
\left(\sum_{i=1}^{n-1}|X_{i}\cdot \boldsymbol{y}|^{2}\right)^{1/2}\leq \eta.
\]
In other words, once we expose the first $n-1$ rows of the matrix, either the matrix cannot be extended to one satisfying the event in \cref{prop:eliminate-large-LCD}, or there is some unit vector $\boldsymbol{y} \in \Gamma^{1}(\eta)$, which can be chosen after looking only at the first $n-1$ rows, and which satisfies the equation above. For the rest of the proof, we condition on the first $n-1$ rows $X_1,\dots,X_{n-1}$ (and hence, a choice of  $\boldsymbol{y}$).

For any vector $\boldsymbol{w'}\in \S^{2n-1}$ with $w'_n \neq 0$, we can write
\[
X_{n}=\frac{1}{w_{n}'}\left(\boldsymbol{u}-\sum_{i=1}^{n-1}w_{i}'X_{i}\right),
\]
where $\boldsymbol{u}:= \boldsymbol{w'}^{T}M_n$.
Thus, restricted to the event $\{s_n(M_n) \leq \eta\}\bigwedge \{\|\boldsymbol{a'}_{M_n}\|_{\infty} = |a'_n|\}$, we have
\begin{align*}
\left|X_{n}\cdot \boldsymbol{y}\right| & =\inf_{\boldsymbol{w'}\in \S^{2n-1}, w'_n \neq 0}\frac{1}{|w_{n}'|}\left|\boldsymbol{u}\cdot \boldsymbol{y}-\sum_{i=1}^{n-1}w_{i}'X_{i}\cdot \boldsymbol{y}\right|\\
 &\leq  
 \frac{1}{|a_{n}'|}\left(\|\boldsymbol{a'}_{M_n}^{T}M_{n}\|_{2}\|\boldsymbol{y}\|_{2}+\|\boldsymbol{a'}_{M_n}\|_{2}\left(\sum_{i=1}^{n-1}|X_{i}\cdot \boldsymbol{y}|^{2}\right)^{1/2}\right)\\
 &\leq \eta \sqrt{n}\left(\|\boldsymbol{y}\|_{2} + \|\boldsymbol{a'}_{M_n}\|_{2}\right) \leq 2\eta \sqrt{n},
\end{align*}
where the second line is due to the Cauchy-Schwarz inequality and the particular choice $\boldsymbol{w'}=\boldsymbol{a'}_{M_n}$.
It follows that the probability in \cref{eqn:intersected-event} is bounded by
$$\rho_{2\eta \sqrt{n},z}(\boldsymbol{y}) \leq 2C_{\ref{thm:LCD-controls-sbp}}\left(n^{3/2}\eta + \exp(-C_{\ref{thm:LCD-controls-sbp}}^{-1}n^{1/50})\right), $$
which completes the proof. 
\end{proof}

\end{document}